\documentclass[11pt,reqno]{amsart}
\usepackage{amsaddr}
\usepackage[ocgcolorlinks,hyperfootnotes=false,colorlinks=true,citecolor=blue,linkcolor=blue,urlcolor=blue]{hyperref}
\author{Tirthankar Bhattacharyya and Abhay Jindal}
\address{Department of Mathematics, 	Indian Institute of Science, \\
	Bangalore 560012, India}
\email{tirtha@iisc.ac.in; abjayj@iisc.ac.in}

\usepackage{color, bm, amscd, tikz-cd}
\usepackage{csquotes}
\newcommand{\bydef}{\stackrel{\rm def}{=}}

\usepackage{tikz}

\newcommand{\cB}{{\mathcal B}}

\newcommand{\cD}{{\mathcal D}}
\newcommand{\cE}{{\mathcal E}}
\newcommand{\cF}{{\mathcal F}}

\newcommand{\cH}{{\mathcal H}}

\newcommand{\cM}{{\mathcal M}}

\newcommand{\cW}{{\mathcal W}}

\newcommand{\la}{\langle}
\newcommand{\ra}{\rangle}

\newcommand{\bfT}{\textit{\textbf{T}}}

\newcommand{\bfM}{\textit{\textbf{M}}}

\newcommand{\bfZ}{\textit{\textbf{Z}}}
\newcommand{\bfW}{\textit{\textbf{W}}}

\newcommand{\bfB}{\textit{\textbf{B}}}
\newtheorem{thm}{Theorem}[section]

\newtheorem{lemma}[thm]{Lemma}

\newtheorem{definition}[thm]{Definition}
\newtheorem{remark}[thm]{Remark}

\numberwithin{equation}{section}

\def\textmatrix#1&#2\\#3&#4\\{\bigl({#1 \atop #3}\ {#2 \atop #4}\bigr)}
\def\dispmatrix#1&#2\\#3&#4\\{\left({#1 \atop #3}\ {#2 \atop #4}\right)}
\numberwithin{equation}{section}

\def\textmatrix#1&#2\\#3&#4\\{\bigl({#1 \atop #3}\ {#2 \atop #4}\bigr)}
\def\dispmatrix#1&#2\\#3&#4\\{\left({#1 \atop #3}\ {#2 \atop #4}\right)}

\begin{document}
	
	\title[Characteristic Functions]{Kernels With Complete Nevanlinna-Pick Factors And The Characteristic Function}
	\maketitle
	\begin{abstract}
The Sz.-Nagy Foias characteristic function for a contraction has had a rejuvenation in recent times due to a number of authors. Such a classical object relates to an object of very contemporary interest, viz., the complete Nevanlinna-Pick kernels. Indeed, a unitarily invariant kernel on the unit ball {\em admits} a characteristic function if and only if it is a complete Nevanlinna-Pick kernel. However, what has captured our curiosity are the recent advancements in constructing characteristic functions for kernels that do not have complete Nevanlinna-Pick property. In such cases, the reproducing kernel Hilbert space which has served as the domain of the multiplication operator has always been the vector-valued Drury-Arveson space (thus the Hardy space in case of the unit disc).

We present a unified framework for deriving characteristic functions for kernels that allow a complete Nevanlinna-Pick factor. Notably, our approach not only encapsulates all previously documented cases but also achieves a remarkable level of generalization, thereby expanding the concept of the characteristic function substantially. We also provide an explanation for the prominence of the Drury-Arveson kernel in all previously established results by showing that the Drury-Arveson kernel was the natural choice inherently suitable for those situations.
\end{abstract}
	
	{\footnotesize \noindent 2020 Mathematics Subject Classification: 47A45, 47A13, 46E22. }
	
	\section{Introduction}
	For an integer $d\geq 1$, a {\em unitarily invariant kernel} on the Euclidean open unit ball $\mathbb{B}_{d}$ in $\mathbb{C}^{d}$ is a function
	\begin{equation*}
		k(\bm{z},\bm{w})=\sum\limits_{n=0}^{\infty}a_{n}^{(k)} \langle \bm{z},\bm{w}\rangle^{n}, \quad \bm{z},\bm{w}\in\mathbb{B}_{d}
	\end{equation*}
	for some sequence of strictly positive numbers $\{a_{n}^{(k)}\}_{n\geq 0}$ such that $a_{0}^{(k)} =1.$ 	The associated reproducing kernel Hilbert space, denoted by $H_{k}$ and referred to as a {\em unitarily invariant space} has the {\em kernel functions}
	$$
	k_{\bm w}(\bm z) = k(\bm{z},\bm{w}) \hspace{5mm} \text{ for } \bm z , \bm{w} \in \mathbb{B}_{d}	
	$$
	as a total subset. We shall assume $k$ to be {\em non-vanishing} in $\mathbb{B}_{d}$. The {\em generalized Bergman kernels}
	\begin{equation}\label{km}
		k_{m}(\bm{z}, \bm{w}) = \left(\frac{1}{1 - \langle \bm{z}, \bm{w} \rangle}\right)^m; \bm{z}, \bm{w} \in \mathbb{B}_{d} \text{ and $m$ is a positive integer }
	\end{equation}
	as well as $k(\bm{z}, \bm{w}) = (\sum_{i=0}^n \langle \bm{z}, \bm{w} \rangle^i) \circ k_{m}(\bm{z}, \bm{w})$ for an integer $n$ where $\circ$ denotes the Schur product are examples of such kernels. When $m=1$, the kernel $k_m$ is known as the {\em Drury-Arveson kernel}.
	
	It is straightforward that there exists a sequence of real numbers $\{b_{n}^{(k)}\}_{n=1}^{\infty}$ such that the equality
	\begin{equation}\label{bn}
		\sum\limits_{n=1}^{\infty} b_{n}^{(k)} \la \bm{z}, \bm{w} \ra^{n} = 1- \frac{1}{\sum\limits_{n=0}^{\infty} a_{n}^{(k)}\la \bm{z}, \bm{w}\ra^{n}} 	
	\end{equation}
	holds for all $\bm{z},\bm{w} \in \mathbb{B}_{d}.$
	
	Denote by $\mathbb{Z}_{+}$ the set of all non-negative integers. Let $\alpha= (\alpha_{1}, \ldots ,\alpha_{d})\in\mathbb{Z}^{d}_{+}$ be a {\em multi-index}. Let $\bm{z}\in \mathbb{C}^{d}$. We need the following notations.
	$$ |\alpha|=\alpha_{1}+ \cdots +\alpha_{d}, \;\; \alpha !=\alpha_{1}! \ldots \alpha_{d}!, \;\; \binom{|\alpha|}{\alpha} = \frac{|\alpha|!}{\alpha_{1}!\dots\alpha_{d}!} \text{ and } \bm z^{\alpha} = z_{1}^{\alpha_{1}}  \ldots z_{d}^{\alpha_{d}}.$$
	To simplify notations, we define the coefficients $a_{\alpha}^{(k)}$ and $b_{\alpha}^{(k)}$ as follows:
	\begin{equation*}
		a_{\alpha}^{(k)} = \begin{cases}  a_{|\alpha|}^{(k)} \binom{|\alpha|}{\alpha}, & \alpha\in\mathbb{Z}^{d}_{+} \\ 0, & \alpha\in\mathbb{Z}^{d} \backslash \mathbb{Z}^{d}_{+} \end{cases},\quad  \text{and} \quad
		b_{\alpha}^{(k)} = b_{|\alpha|}^{(k)} \binom{|\alpha|}{\alpha},\quad  \alpha\in\mathbb{Z}^{d}_{+} \backslash \{0\}.
	\end{equation*}
	By a combinatorial argument, it follows that $k(\bm{z},\bm{w}) = \sum\limits_{\alpha\in\mathbb{Z}^{d}_{+}} a_{\alpha}^{(k)} \bm{z}^{\alpha} \overline{\bm{w}^{\alpha}}$. It is evident that the monomials $\{\bm z^{\alpha}\}_{\alpha\in\mathbb{Z}^{d}_{+}}$ serve as an orthogonal basis for the unitarily invariant space $H_{k}$. Furthermore, $ a_{\alpha} \|\bm z^{\alpha}\|^{2}_{H_{k}} = 1$ for all $\alpha\in\mathbb{Z}^{d}_{+}.$
	
	For a Hilbert space $\cE$,  the vector-valued Hilbert space $H_{k}(\cE)$ is defined as the class of all holomorphic $\cE$-valued functions on $\mathbb{B}_{d}$ whose Taylor expansion $f(\bm z) = \sum\limits_{\alpha\in\mathbb{Z}^{d}_{+}} c_{\alpha} \bm z^{\alpha}, c_{\alpha} \in \cE $ satisfies $\|f\|^{2}:= \sum\limits_{\alpha\in\mathbb{Z}^{d}_{+}} \|c_{\alpha}\|^{2} \|\bm z^{\alpha}\|^{2} < \infty  $.
	As a Hilbert space, $H_{k}(\cE)$ is equivalent to $H_{k} \otimes \cE,$ with the identification given by
	$$ \sum\limits_{\alpha\in\mathbb{Z}^{d}_{+}} c_{\alpha} \bm z^{\alpha} \rightarrow  \sum\limits_{\alpha\in\mathbb{Z}^{d}_{+}} (\bm{z}^{\alpha} \otimes c_{\alpha}). $$
	
	We use the notation $\bfT$ to denote a commuting $d$-tuple of bounded operators $(T_{1}, \dots, T_{d})$. The existence of a characteristic function for a {\em pure $1/k$-contractive} (to be defined in the context) operator tuple $\bfT$ will be shown in Section \ref{GeneralCHFN} for a certain class of kernels. In that section, we shall also answer a couple of natural questions related to the existence. Since the construction by Sz.-Nagy and Foias is explicit and relates to several important issues like the realization formula for example, we feel that mere existence is not enough. Hence, the  characteristic function will be constructed in Section \ref{construction} following an idea of \cite{Esch} which we recall for its influence on  further work \cite{BDS} as well as on the present paper.  An operator valued function $\theta: \mathbb{B}_{d} \to \mathcal{B}(\cE, \cF)$ is called {\em $k-$inner} if the induced map $\cE \to H_{k}(\cF)$ sending a vector $x$ of $\cE$ to the function $\bm z \mapsto \theta (z)x$ is isometric and the range of the induced map is orthogonal to $\bfM_{\bm z }^{\alpha} (\theta \cE)$ for all $\alpha $ with $|\alpha| \geq 1.$ Olofsson \cite{Olof}, Eschmeier \cite{Esch} and Eschmeier--Toth \cite{ES} studied $k-$inner functions, progressing in an increasing order of generality. Eschmeier--Toth developed a realization formula for such functions as well as associated a canonical $k-$inner function with every pure $1/k-$contraction for a special class of kernels $k$. This function is closely related to the characteristic function that we shall develop. We shall comment more on this connection in Remark \ref{K-inner}.

	\section{Generalization of the characteristic function} \label{GeneralCHFN}
	As we mentioned before, this section will show the existence of a characteristic function for a certain class of operators related to a certain class of kernels.
	
	\subsection{The operators}
	For a given multi-index $\alpha\in\mathbb{Z}^{d}_{+},$ the notation $\bfT^{\alpha}$ will stand for $T_{1}^{\alpha_{1}} \dots T_{d}^{\alpha_{d}}$. The following definition, which is inspired by the expression in \eqref{bn} within the framework of a unitarily invariant kernel on $\mathbb{B}_{d}$, is introduced in \cite{CH} and plays a vital role in our analysis.
	\begin{definition}
		For a unitarily invariant kernel $k$ on $\mathbb{B}_{d}$ and a commuting $d$-tuple of bounded operators $\bfT$, if the series $\sum\limits_{\alpha\in\mathbb{Z}^{d}_{+} \backslash \{0\}} b_{\alpha}^{(k)} \bfT^{\alpha}(\bfT^\alpha)^{*}$ converges in the strong operator topology to a contraction, then the $d-$tuple $\bfT$ is referred to as a $1/k$-contraction.
		In this case, we denote the unique positive square root of the positive operator $ I-\sum\limits_{\alpha\in\mathbb{Z}^{d}_{+} \backslash \{0\}} b_{\alpha}^{(k)} \bfT^{\alpha}(\bfT^{\alpha})^{*}$ by $\Delta_{\bfT}.$
	\end{definition}
	In the case of the Szeg\H{o} kernel on the unit disk $\mathbb{D}$, a contraction $T$ is indeed a $1/k$-contraction. For the Drury-Arveson kernel on $\mathbb{B}_{d}$, a {\em commuting contractive tuple} $\bfT = (T_{1}, \ldots, T_{d})$ is a $1/k$-contraction. This notion is also referred to as a $d$-contraction in Arveson's terminology, as mentioned on page 175 of \cite{curv}.

The idea for $1/k$-contractions that are sometimes also called $k$-contractions goes back to Agler (see for example \cite{Agler1} and \cite{Agler2}). The idea has also been developed by Pott \cite{Pott}. Eschmeier--Toth in \cite{ES} and Schillo in  \cite{Sch} give a unified approach to $1/k$-contractions for large class of kernels $k$.  See also Abadias--Bello--Yakubovich \cite{ABY} and Clouatre--Timko \cite{CT} for various aspects of the topic.

	\begin{definition}\label{pure}
	A $1/k-$contraction $\bfT = (T_{1}, \hdots,T_{d})$ is called pure if the series  $$\sum\limits_{\alpha\in\mathbb{Z}^{d}_{+}} a_{\alpha}^{(k)} \bfT^{\alpha} \Delta_{\bfT}^{2}(\bfT^{\alpha})^{*}$$
	converges strongly to the identity operator $I.$
	\end{definition}
	
	The concept of purity for $1/k$-contractions is well-established and originates from the works of Ambrozie--Englis--Muller \cite{AEM} and Arazy--Englis \cite{AE}. It has been extensively studied in the literature. Let $\bfT = (T_{1}, \dots, T_{d})$ be a tuple of bounded operators on a Hilbert space $\cH.$ Let $\cM$ be closed subspace of $\cH.$ Then $\cM$ is an invariant subspace for $\bfT$ if $T_{i} \cM \subseteq \cM$ for all $i$, and $\cM$ is a co-invariant subspace for $\bfT$ if $T_{i}^{*} \cM \subseteq \cM$ for all $i.$  One important property of pure $1/k$-contractions is that their compressions to co-invariant subspaces also yield pure $1/k$-contractions. In other words, if $\bfT = (T_{1}, \ldots, T_{d})$ is a pure $1/k$-contraction, and $P_{\cH}\bfT|_{\cH} = (P_{\cH}T_{1}|_{\cH}, \dots, P_{\cH}T_{d}|_{}\cH)$ is the compression of $\bfT$ to a co-invariant subspace $\cH$, then $P_{\cH}\bfT|_{\cH}$ is also a pure $1/k$-contraction. The next tool is a construction introduced by Arazy and Englis in \cite{AE}, which has its roots in \cite{AEM}. The following theorem, which we will refer to, can be found in Theorem 1.3 of \cite{AE}.
	\begin{thm} \label{V_T}
		Suppose we have a pure $1/k$-contraction $\bfT = (T_{1}, \ldots,T_{d})$ acting on a Hilbert space $\cH$.  Then the linear map $V_{\bfT}:\cH \to H_{k} \otimes \overline{\rm Ran} \Delta_{\bfT}$ given by
		$$h \mapsto \sum\limits_{\alpha\in\mathbb{Z}^{d}_{+}} a_{\alpha}^{(k)} \bm{z}^{\alpha} \otimes \Delta_{\bfT}(\bfT^{\alpha})^{*}h$$ is an isometry. Moreover it satisfies the relation $$V_{\bfT}^{*}(M_{z_{i}}^{(k)}\otimes I_{\overline{\rm Ran} \Delta_{\bfT}}) = T_{i} V_{\bfT}^{*}$$ for all $i=1,\dots,d.$
	\end{thm}

	\subsection{The kernels}
	Let $\cE$ and $\cF$ be two Hilbert spaces. A $\cB(\cE, \cF)-$valued function $\varphi$ defined on $\mathbb{B}_{d}$ is called a {\em multiplier} from $H_{k} \otimes \cE$ to $H_{k} \otimes \cF$ if the function $\varphi f$ is in $H_{k}\otimes \cF$ for any $f$ in $H_{k}\otimes \cE$. We denote the linear operator of multiplication by $\varphi$ as $M_{\varphi}$. Using the Closed Graph Theorem, $M_{\varphi}$ is a bounded operator. The set of all such functions $\varphi$ is denoted by $Mult(H_{k} \otimes \cE, H_{k} \otimes \cF)$.
	
	In order for the multiplication operators by coordinate functions $z_i$ to serve as the model operators, it is essential to impose restrictions on the class of kernels under consideration. The choice of kernels are assumed to satisfy:
	\begin{enumerate}
		\item The multiplication operators $M_{z_i}^{(k)}$ by the coordinate functions $z_i$ are bounded operators on $H_k$ for any $i=1,\dots,d.$
		\item The tuple $\bfM_{z}^{(k)} \bydef (M_{z_{1}}^{(k)}, \dots, M_{z_{d}}^{(k)})$ is a $1/k$-contraction.
	\end{enumerate}
	Kernels that satisfy these conditions are referred to as {\em admissible kernels}. 
	
	For $(1)$, it is necessary and sufficient that 
	$$\sup\limits_{n \in \mathbb{N}} \left(\frac{a_{n}^{(k)}}{a_{n+1}^{(k)}} \right) < \infty.$$
	For $(2),$ Proposition 2.1 and Lemma 2.2 in \cite{Chen} give a sufficient condition. The generalized Bergman kernels $k_{m}$ defined in \eqref{km} serve as examples of admissible kernels, as shown in \cite{MV}. More generally, any finite product of {\em unitarily invariant CNP kernels} (to be defined later in this section) is an admissible kernel, see Proposition 5.4 in \cite{O}. For an admissible kernel $k$, it is known that $\Delta_{\bfM_{\bm{z}}} = E_{0},$ where $E_{0}$ denotes the projection of $H_{k}$ onto the subspace of constant functions. Furthermore, the operator tuple $\bfM_{\bm{z}}^{(k)}$ is pure in the sense of Definition \ref{pure}, as stated in Lemma 2.1 in \cite{BJ}.
	
	Let $T$ be a contraction on a Hilbert space. The Sz.-Nagy Foias characteristic function $\theta_T$ of $T$ is given by
	$$\theta_{T}(z) = -T + z D_{T^{*}} (I - z T^{*})^{-1} D_{T}, \quad z \in \mathbb{D}$$
	where $D_T$ and $D_{T^{*}}$ are the classical defect operators. It induces a multiplier from $H^2 \otimes \mathcal D_T$ into $H^2 \otimes \mathcal D_{T^*}$ which factors the projection $ I - V_{\bfT} V_{\bfT}^{*}$ (onto a shift-invariant subspace) in the sense that $ I - V_{\bfT} V_{\bfT}^{*} = M_{\theta_T}M_{\theta_T}^*$. It follows from Halmos's uniqueness theorem for Lax's characterization of invariant subspaces - see Theorem 4 in \cite{Halmos} - that up to multiplication by a partial isometry, $\theta_T$ is the only function which does this factoring. This characterization of $\theta_T$ in terms of the projection $ I - V_{\bfT} V_{\bfT}^{*}$ led us to define the notion of a characteristic function in a much more general setting, see Definition 2.5 in \cite{BJ}. It then turns out that existence of a characteristic function is necessary and sufficient for the kernel to be a complete Nevanlinna-Pick kernel, see Theorem 3.43 in \cite{BJ}.

	\begin{definition}
		A reproducing kernel $k$ on the open unit ball $\mathbb{B}_{d}$ is said to have the $M_{m\times n}$ Nevanlinna-Pick property if, for any set of points $\bm{\lambda}_{1}, \ldots ,\bm{\lambda}_{N}$ in $\mathbb{B}_{d}$ and $m$-by-$n$ matrices $W_{1}, \ldots ,W_{N}$ satisfying the condition
		$$(I-W_{i}W_{j}^{*}) k(\bm{\lambda}_{i}, \bm{\lambda}_{j}) \geq 0,$$
		there exists a multiplier $\varphi$ in the closed unit ball of $$Mult(H_{k} \otimes \mathbb{C}^{n}, H_{k} \otimes \mathbb{C}^{m})$$ such that $\phi(\bm{\lambda}_{i})=W_{i},i=1, \ldots ,N.$ If the reproducing kernel $k$ has the $M_{m \times n}$ Nevanlinna-Pick property for all positive integers $m$ and $n$, it is said to have the complete Nevanlinna-Pick property. In this case, the corresponding reproducing kernel Hilbert space is called a complete Nevanlinna-Pick space.
	\end{definition}
	
	\begin{definition}
		A reproducing kernel Hilbert space $H_{k}$ is said to be irreducible if the reproducing kernel $k(\bm z, \bm w)$ is nonzero for all $\bm z, \bm w$ in the open unit ball $\mathbb{B}_{d}$, and the kernel functions $k_{\bm w}$ and $k_{\bm v}$ are linearly independent whenever $\bm v \neq \bm w$.
	\end{definition}
	
	One of the main tools we will utilize is a well-known result, which can be found in Lemma 2.3 of \cite{H} (or Lemma 7.33 of \cite{AM} in the one-variable case).
	
	\begin{lemma}
		Let $H_{k}$ be a unitarily invariant space on $\mathbb{B}_{d}$ with reproducing kernel given by
		$$ k(\bm{z},\bm{w}) = \sum\limits_{n=0}^{\infty}a_{n}^{(k)}\langle \bm{z}, \bm{w} \rangle^{n}.$$
		Then the following are equivalent:
		\begin{enumerate}
			\item $H_{k}$ is an irreducible complete Nevanlinna-Pick space.
			\item The sequence $\{b_{n}^{(k)}\}_{n=1}^{\infty}$ defined by (\ref{bn}) is a sequence of non-negative real numbers.
		\end{enumerate}
	\end{lemma}
	
	\begin{definition}
		A reproducing kernel $s$ is called a unitarily invariant complete Nevanlinna-Pick (CNP) kernel if it satisfies the following conditions:
		\begin{enumerate}
			\item The kernel $s$ is of the form $$s(\bm{z}, \bm{w}) = \sum\limits_{n=0}^{\infty} a_{n}^{(s)} \la \bm{z}, \bm{w} \ra^{n} \hspace{5mm} (\bm{z} , \bm{w} \in \mathbb{B}_{d})$$ for a  sequence of strictly positive coefficients $\{a_{n}^{(s)}\}_{n \geq 0}$ with $a_{0}^{(s)} =1.$
			\item $H_{s}$ is an irreducible complete Nevanlinna-Pick space.
		\end{enumerate}
	\end{definition}
	
	It is worth noting that unitarily invariant CNP kernels are admissible. This property is stated in Lemma 5.2 of \cite{CH}. In this paper, the symbol $s$ will consistently represent a unitarily invariant CNP kernel. The $1/s$-contractions for CNP kernels $s$ have been treated in \cite{BJ} and \cite{CH} and many other papers. 
	
	In general, the generalized Bergman kernels $k_{m}$ defined in \eqref{km} are not always unitarily invariant CNP. Some basic examples of unitarly invariant CNP kernels are the Drury-Arveson kernel and the Dirichlet kernel.
	
	While the generalized Bergman kernels $k_m$ defined in \eqref{km} serve as examples of admissible kernels, they are not CNP kernels except for $m=1$. However, there are other well-known examples of unitarily invariant CNP kernels. Two notable examples are the Drury-Arveson kernel and the Dirichlet kernel. These kernels have been extensively studied and play a fundamental role in the theory of unitarily invariant spaces and operator theory.

	In spite of the fact that existence of a characteristic function (according to Definition 2.5 in \cite{BJ}) is a characterization of CNP kernels, there has been considerable development in the last few years to define a characteristic function when $k$ is not a CNP kernel, see \cite{BDS} when $k$ is a generalized Bergman kernel in several variables and see \cite{BDRS} when $k$ is a kernel associated with so-called $w$-hypercontractions. Of course, this means then that the characteristic function cannot be a multiplier between vector valued reproducing kernel Hilbert spaces with the kernel being $k$ at both the domain and the range. We put together a general framework below which encompasses all these works.
	
	\subsection{The Characteristic Function}
	
	\begin{definition}
		Let $k$ and $l$ be admissible kernels on $\mathbb{B}_{d}.$ A pure $1/k-$contraction $\bfT = (T_{1}, \dots, T_{d})$ is said to admit a characteristic function through the kernel $l$ if there exist a Hilbert space $\cE$ and $\cB(\cE, \overline{\rm Ran} \Delta_{\bfT})-$valued analytic function $\theta_{\bfT}$ defined on $\mathbb{B}_{d}$ such that $M_{\theta_{\bfT}}$ is a multiplication operator from $H_{l} \otimes \cE$ to $H_{k} \otimes \overline{\rm Ran} \Delta_{\bfT}$ satisfying $$ I - V_{\bfT} V_{\bfT}^{*} = M_{\theta_{\bfT}} M_{\theta_{\bfT}}^{*}$$
		where $V_{\bfT}$ is the isometry defined in Theorem \ref{V_T}.
	\end{definition}
	
	This generalized definition allows for the consideration of different admissible kernels $k$ and $l$ and provides a framework to investigate the existence of characteristic functions for pure $1/k$-contractions, even in cases where $k$ is not a unitarily invarint CNP kernel. Let us recall Definition 2.7 of the associated tuple from \cite{BJ}:
	\begin{definition}
		Let $\bfT = (T_{1}, \dots, T_{d})$ be a pure $1/k-$contraction. The associated tuple of commuting operators $\bfB_{\bfT}$ is defined on ${\rm Ker} V_{\bfT}^{*}$ as $$ \bfB_{\bfT} = ((M_{z_{1}}^{(k)} \otimes I)|_{{\rm Ker} V_{\bfT}^{*}}, \dots, (M_{z_{d}}^{(k)} \otimes I)|_{{\rm Ker} V_{\bfT}^{*}}).$$
	\end{definition}
	
	The proof of the following theorem follows a similar line of reasoning as presented in the proof of Theorem 2.8 in \cite{BJ}. It establishes the equivalence between the existence of a characteristic function for $\bfT$ through the kernel $l$ and the associated tuple $\bfB_{\bfT}$ being a $1/l$-contraction. This result highlights the close connection between characteristic functions and the associated tuples.
	\begin{thm}\label{characterization2}
		Let $k$ and $l$ be admissible kernels on $\mathbb{B}_{d}.$ A pure $1/k-$contraction $\bfT$ admits a characteristic function through the kernel $l$ if and only if the associated tuple $\bfB_{\bfT}$ is a $1/l-$contraction.
	\end{thm}
	
	The following definition presents a generalization of Definition 3.1 in \cite{BJ}:
	\begin{definition}
		Let $k$ and $l$ be admissible kernels on $\mathbb{B}_{d}.$ The kernel $k$ is said to admit a characteristic function through the kernel $l$ if every pure $1/k-$contraction admits a characteristic function through the kernel $l.$
	\end{definition}
	This notion allows for the examination of the existence of characteristic functions across a broader range of kernels and their relationships.
	
	\begin{lemma} \label{div}
		Assuming $k$ and $l$ are admissible kernels and $l(\bm{z}, \bm{w}) \neq 0$ for all $\bm{z}, \bm{w} \in \mathbb{B}_d$, the operator tuple $\bfM_{\bm{z}}^{(k)} = (M_{z_{1}}^{(k)}, \dots, M_{z_{d}}^{(k)})$ is a $1/l$-contraction if and only if the kernel $k/l$ is positive semi-definite.
	\end{lemma}
	\begin{proof}
		Given that the kernel $l$ is non-vanishing on $\mathbb{B}_{d} \times \mathbb{B}_{d}$, the equality \eqref{bn} holds for all $\bm{z}, \bm{w} \in \mathbb{B}_{d}$. Since the set of kernel function is total, for showing that the operator 
		$$ I - \sum\limits_{\alpha\in\mathbb{Z}^{d}_{+} \backslash \{0\}} b_{\alpha}^{(l)} \bfM_{\bm{z}}^{(k)\alpha} (\bfM_{\bm{z}}^{(k)\alpha})^{*} $$
	   is positive, it is enough to test on linear combinations of kernel functions. But this is equivalent to the fact that the function 
	   $$  \frac{k(\bm z, \bm w)}{l(\bm z, \bm w)} =  \left( 1 - \sum\limits_{\alpha\in\mathbb{Z}^{d}_{+} \backslash \{0\}} b_{\alpha}^{(l)} \bm{z}^{\alpha} \overline{\bm{w}}^{\alpha} \right) k(\bm{z},\bm{w} )$$
		is positive semi-definite on $\mathbb{B}_{d} \times \mathbb{B}_{d}$.
	\end{proof}
	
	The proof of the following theorem can be obtained using Lemma \ref{div} as a key tool. The overall approach is similar to the proof presented in Theorem 3.2 of \cite {BJ}.
	\begin{thm} \label{ksg}
		If $k = sg,$ where $s$ is a unitarily invariant CNP kernel and $g$ is a positive kernel, then $k$ admits a characteristic function through the kernel $s$.
	\end{thm}
	
	Theorem \ref{ksg} naturally leads to the question of whether $k$ admits a characteristic function through a kernel other than $s$ as well. In the case when $s$ is the Drury-Arveson kernel and $k = s^m$ for $m = 2, 3, \ldots$, we have the answer.
	
	\begin{thm}
		If the kernel $k_{m}$ admits a characteristic function through the kernel $k_{n}$ then $n=1.$
	\end{thm}
	\begin{proof}
		The expression for the generalized Bergman kernel $k_{m}$ is given by
		$$k_m(\bm{z}, \bm{w}) = \sum\limits_{\alpha\in\mathbb{Z}^{d}_{+}} \sigma_{m}(\alpha) \bm{z}^{\alpha} \overline{\bm{w}^{\alpha}}, \quad \bm{z} , \bm{w} \in \mathbb{B}_{d},$$
		where
		$$ \sigma_{m}(\alpha) = \frac{(m+ |\alpha| -1) !}{\alpha! (m-1)!},	\quad \alpha\in\mathbb{Z}^{d}_{+}.$$
		Recall that $\{ e_{m}(\alpha) = \sqrt{\sigma_{m}(\alpha)} \bm{z}^{\alpha} : \alpha\in\mathbb{Z}^{d}_{+}\}$ is an orthonormal basis. For $N \geq 0,$ the subspace $\cH_{N} $ is defined as the span of monomials $\bm{z}^{\alpha}$ with $|\alpha| \leq N$.  We define an operator tuple $\bfT_{N}$ acting on $\cH_{N}$ as $\bfT_{N} = P_{\cH_{N}} \bfM_{\bm{z}}^{(k_{m})}|_{\cH_{N}},$ where $P_{\cH_{N}}$ is the projection onto $\cH_{N}.$ It is easy to check that the $d-$tuple $\bfT_{N}$ is a pure $1/k_{m}-$contraction, $\Delta_{\bfT_{N}} = E_0$ and $$\cM_{N} \bydef {\rm Ker} V_{\bfT_{N}}^{*} = \overline{\rm span}\{ \bm{z}^{\alpha} : |\alpha| \geq N+1\}.$$ Let $P_{\cM_{N}}$ be the projection onto $\cM_{N}.$ Let $\bm n$ denote the multi-index $(n,0, \dots, 0).$
		
		Since the kernel $k_{m}$ admits a characteristic function through the kernel $k_{n}$, we can apply Theorem \ref{characterization2} to obtain that the operator tuple $\bfM_{\bm{z}}^{(k_{m})}|{\cM_{N}}$ is a $1/k_{n}$-contraction.  This implies that for all $N \geq 0,$ we have
		\begin{equation} \label{ineq} \Big\langle \Big(I_{\cM_{N}} - \sum\limits_{\alpha\in\mathbb{Z}^{d}_{+} \backslash \{0\}} b_{\alpha}^{(n)}\bfM_{\bm{z}}^{(k_{m}) \alpha} P_{\cM_{N}} (\bfM_{\bm{z}}^{(k_{m})\alpha})^{*} \Big) e_{m}(\bm{N+2}), e_{m}(\bm{N+2}) \Big\rangle \geq 0.
		\end{equation}
		The values of $b_{\alpha}^{(m)}$ for the generalized Bergman kernel $k_m$ are given by:
		$$b_{\alpha}^{(m)} = \begin{cases} (-1)^{|\alpha| +1} \binom{m}{|\alpha|} \binom{|\alpha|}{\alpha}, & 1 \leq |\alpha| \leq m \\ 0,  & |\alpha| > m \end{cases}.$$
		The inequality in \eqref{ineq} now implies that for all $N \geq 0,$ we have
		$$ 1 - b_{\bm{1}}^{(n)} \frac{\sigma_{m}(\bm{N+1})}{\sigma_{m}(\bm{N+2})} \geq 0.$$
		In other words,
		$$1 - n \frac{(m+N)!}{(m-1)! (N+1)!} \frac{(m-1)! (N+2)!}{ (m+N+1)!} \geq 0.$$ After simplification we get
		$$ n \leq \frac{N+m+1}{N+2}.$$ Since this holds for all $N\geq 0,$ we get that $n=1.$
	\end{proof}
	
	If the kernel $k$ admits two factorizations, i.e., $ k = s_{1} g_{1} = s_{2} g_{2} $
	where $s_{1}$ and $s_{2}$ are unitarily invariant CNP kernels, and $g_{1}$ and $g_{2}$ are positive semi-definite kernels, then every pure $1/k-$contraction admits a characteristic function through $s_{1}$ and another through $s_{2}$. Is there a relation between them as in Theorem 4.2 of \cite{McT}?
	\begin{thm}\label{s_1 = s_2}
		Suppose $k$ is an admissible kernel with two factorizations $k=s_{1}g_{1} = s_{2} g_{2}$ as above. Let $\bfT = (T_{1}, \dots, T_{d})$ be a non-trivial pure $1/k-$contraction and $V_{\bfT}$ be the isometry defined in Theorem \ref{V_T}. Let $\theta_{1} \in {\rm Mult}(H_{s_{1}} \otimes \cE_{1}, H_{k} \otimes \overline{\rm Ran} \Delta_{\bfT})$ and $\theta_{2} \in {\rm Mult}(H_{s_{2}} \otimes \cE_{2}, H_{k} \otimes \overline{\rm Ran} \Delta_{\bfT})$ be two partially isometric multipliers such that
		$$ I - V_{\bfT} V_{\bfT}^{*} = M_{\theta_{i}} M_{\theta_{i}}^{*}, \quad i =1,2. $$
		Then there exists a partial isometry $V_m : H_{s_{1}} \otimes \cE_{1} \to H_{s_{2}} \otimes \cE_{2}$ such that
		$$ M_{\theta_{1}} = M_{\theta_{2}} V_m \quad \text{and} \quad M_{\theta_{2}} = M_{\theta_{1}}V_m^{*} $$
		Moreover, there exists a partial isometry $V_f : \cE_{1} \to \cE_{2}$ such that
		$$ \theta_{1}(\bm z) = \theta_{2}(\bm z) V_f \quad \text{and} \quad \theta_{2}(\bm z) = \theta_{1}(\bm z)V_f^{*} $$
		for all $\bm z\in\mathbb{B}_{d}$ if and only if $s_{1} = s_{2}.$
	\end{thm}
	
	\begin{proof}
		The word non-trivial in the statement means that $\bfT $ is not unitarily equivalent to $\bfM_{\bm z}^{(k)}.$ We have
		$$ M_{\theta_{1}} M_{\theta_{1}}^{*} = M_{\theta_{2}} M_{\theta_{2}}^{*}.$$
		This implies that
		$$ \la s_{1 \bm w} \otimes \theta_{1}(\bm w)^{*} \eta , s_{1 \bm z} \otimes \theta_{1}(\bm z)^{*} \xi \ra  = \la s_{2 \bm w} \otimes \theta_{2}(\bm w)^{*} \eta , s_{2 \bm z} \otimes \theta_{2}(\bm z)^{*} \xi \ra$$
		for all $\bm z, \bm w \in \mathbb{B}_{d}$ and $\eta, \xi \in \overline{\rm Ran} \Delta_{\bfT}.$ Hence there exists a unitary operator
		$$V_m: \bigvee\limits_{\substack {   \xi \in \overline{\rm Ran} \Delta_{\bfT} \\ \bm z\in \mathbb{B}_{d} }} s_{1 \bm z} \otimes \theta_{1}(\bm z)^{*} \eta \to \bigvee\limits_{\substack { \xi \in \overline{\rm Ran} \Delta_{\bfT}  \\ \bm z\in \mathbb{B}_{d} }} s_{2 \bm z} \otimes \theta_{2}(\bm z)^{*} \eta$$
		given by
		$$ s_{1 \bm z} \otimes \theta_{1}(\bm z)^{*} \eta \mapsto  s_{2 \bm z} \otimes \theta_{2}(\bm z)^{*} \eta$$
		where $\bigvee$ denotes the closed linear span. Extending $V_m$ by zero on the orthogonal complement, we obtain a partial isometry from $H_{s_{1}} \otimes \cE_{1}$ into $H_{s_{2}} \otimes \cE_{2}$, which we continue to denote by $V_m.$ The partial isometry $V_m$ satisfies
		$$ M_{\theta_{1}}^{*} = V_m^{*} M_{\theta_{2}}^{*}  \quad \text{and} \quad  M_{\theta_{2}}^{*} =  V_m M_{\theta_{1}}^{*}.$$
		
		For the moreover part, it is straightforward to see that $s_1=s_2$ would give such a $V_f$, see also Proposition 2.3 in \cite{CHS}. For the converse, by virtue of $M_{\theta_{1}} M_{\theta_{1}}^{*} = M_{\theta_{2}}M_{\theta_{2}}^{*},$ we have
		$$ s_{1}(\bm z, \bm z) \la \theta_{1}(\bm z)^{*} \eta, \theta_{1}(\bm z)^{*} \eta \ra = s_{2}(\bm z, \bm z) \la \theta_{2}(\bm z)^{*} \eta, \theta_{2}(\bm z)^{*} \eta \ra$$
		for all $\bm z\in \mathbb{B}_{d}$ and $\eta \in \overline{\rm Ran} \Delta_{\bfT}.$ Using $\theta_{2}(\bm z)^{*} = V_f \theta_{1}(\bm z)^{*}$ and $\theta_{1}(\bm z) V_f^{*} V_{f} = \theta_{1}(\bm z)$, we get that
		$$\left(s_{1} (\bm z, \bm z) - s_{2} (\bm z, \bm z) \right) \left\| \theta_{1}(\bm z)^{*} \eta \right\|^{2} = 0 $$
		for all $\bm z\in \mathbb{B}_{d}$ and $\eta \in \overline{\rm Ran} \Delta_{\bfT}.$ Suppose for some $\bm z_{0} \in \mathbb{B}_{d},$ we have $\theta_{1}(\bm z_{0})^{*} \eta = 0 $ for all $\eta \in \overline{\rm Ran} \Delta_{\bfT}.$ This implies that $\theta_{1}(\bm z_{0}) = 0.$  If the set $\{ \bm z \in \mathbb{B}_{d}; \quad \theta_{1}(\bm z) = 0\}$ contains an open subset of $\mathbb{B}_{d},$ then $\theta_{1}(\bm z) \equiv 0.$ This will imply that $V_{\bfT}$ is a unitary and hence $\bfT$ is unitarily equivalent to $\bfM_{\bm z}^{(k)}.$ Thus, if $\bfT$ is not unitarily equivalent to $\bfM_{\bm z}^{(k)},$ then the set $\{ \bm z \in \mathbb{B}_{d}; \theta_{1}(\bm z) = 0\}$ can not contain any open subset of $\mathbb{B}_{d}.$ This implies that the set $\{\bm z\in\mathbb{B}_{d}; s_{1}(\bm z ,\bm z) = s_{2}(\bm z, \bm z)\}$ is dense in $\mathbb{B}_{d}$ and hence $s_{1}(\bm z, \bm z) = s_{2}(\bm z, \bm z)$ for all $\bm z\in\mathbb{B}_{d}.$ Thus, $a_{n}^{(k_{1})} = a_{n}^{(k_{2})}$ for all $n \geq 0.$ This implies that $s_{1} = s_{2}.$
	\end{proof}
	A partial isometry $V_f$ as above at the function level would give rise to a partial isometry $V_m$ as above at the multiplier level. However, as we saw, the converse requires $s_1$ to be equal to $s_2$.
	
	\section{The Construction} \label{construction}
	In this section, we provide an explicit construction for the characteristic function associated with a $1/k-$contraction where $k$ is an admissible kernel that possesses a complete Nevanlinna-Pick (CNP) factor.
	\subsection{Kernels with a CNP factor}
	We shall be concerned with admissible kernels {\em $k$ which have a CNP factor}, i.e.,  $k=s \cdot g$ for a unitarily invariant CNP kernel $s$ and a positive kernel $g.$   The kernel $g$ is defined as follows:
	$$g(\bm{z},\bm{w})=\sum\limits_{n=0}^{\infty}a_{n}^{(g)} \langle \bm{z},\bm{w}\rangle^{n}, \quad \bm{z},\bm{w}\in\mathbb{B}_{d}$$
	for some sequence of non-negative coefficients $\{a_{n}^{(g)}\}_{n\geq 0}$ such that $a_{0}^{(g)} =1$. The coefficients of the kernel $k$ can be related to the coefficients of the CNP kernel $s$ and the positive kernel $g$ through the relationship:
	\begin{equation*}
		a_{n}^{(k)} = \sum\limits_{i=0}^{n} a_{n-i}^{(s)} a_{i}^{(g)}, \quad n \geq 0.
	\end{equation*}
	It follows that $a_{n}^{(s)} \leq a_{n}^{(k)}$ and $a_{n}^{(g)} \leq a_{n}^{(k)}.$ Therefore, if the power series $\sum_{n=0}^{\infty} a_{n}^{(k)} t^{n}$ has radius of convergence $r \geq 1,$ then both the power series $\sum_{n=0}^{\infty} a_{n}^{(s)} t^{n}$ and $\sum_{n=0}^{\infty}a_{n}^{(g)} t^{n}$ will also have radius of convergence greater than or equal to $r.$ In this section, $k$ is assumed to be an admissible kernel with a CNP factor $s$. We further assume that the power series $\sum_{n=0}^{\infty} a_{n}^{(s)} t^{n},$ associated with the CNP factor, has a radius of convergence $1.$ This implies that the power series $\sum_{n=0}^{\infty} a_{n}^{(k)} t^{n},$ associated with the kernel $k,$ also has a radius of convergence $1.$
	
	\begin{lemma}\label{spure}
		The operator $d-$tuple $\bfM_{\bm{z}}^{(k)} = (M_{z_{1}}^{(k)}, \dots, M_{z_{d}}^{(k)})$ acting on the reproducing kernel Hilbert space $H_{k}$ is a pure $1/s$-contraction.
	\end{lemma}
	\begin{proof}
		As a consequence of Lemma \ref{div}, it follows that the operator $d$-tuple $\bfM_{\bm{z}}^{(k)}$ is a $1/s$-contraction. To prove that the $1/s$-contraction $\bfM_{\bm{z}}^{(k)}$ is pure, we aim to show that the series
		the series $$ \sum\limits_{\alpha\in\mathbb{Z}^{d}_{+}} a_{\alpha}^{(s)} \bfM_{\bm{z}}^{(k) \alpha} \Gamma_{\bfM_{\bm{z}}^{(k)}}^{2} (\bfM_{\bm{z}}^{(k) \alpha})^{*}$$ converges strongly to the identity operator on $H_{k},$ where
		$$\Gamma_{\bfM_{\bm{z}}^{(k)}}^{2} = I - \sum\limits_{\alpha\in\mathbb{Z}^{d}_{+} \backslash \{0\}} b_{\alpha}^{(s)} \bfM_{\bm{z}}^{(k)\alpha} (\bfM_{\bm{z}}^{(k)\alpha})^{*}.$$
		To show that $\bfM_{\bm z}^{(k)}$ is pure, consider the following expression for $\bm{z} , \bm{w} \in \mathbb{B}_{d}$:
		\begin{align*}
			\Big\la \Big( \sum\limits_{\alpha\in\mathbb{Z}^{d}_{+}} a_{\alpha}^{(s)} \bfM_{\bm{z}}^{(k) \alpha} \Gamma_{\bfM_{\bm{z}}^{(k)}}^{2} (\bfM_{\bm{z}}^{(k) \alpha})^{*} \Big) k_{\bm{w}} , k_{\bm{z}} \Big\ra & =   \sum\limits_{\alpha\in\mathbb{Z}^{d}_{+}} a_{\alpha}^{(s)}\bm{z}^{\alpha} \overline{\bm{w}}^{\alpha} \la \Gamma_{\bfM_{\bm{z}}^{(k)}}^{2} k_{\bm{w}} , k_{\bm{z}} \ra \\
			& = s(\bm{z}, \bm{w}) g(\bm{z}, \bm{w}) \\
			&= k(\bm{z}, \bm{w}) = \langle k_{\bm{w}}, k_{\bm{z}} \rangle.
		\end{align*}
		This proves the lemma.
	\end{proof}
	
	\begin{lemma}\label{k implies s}
		A pure $1/k-$ contraction is a pure $1/s-$contraction. 	
	\end{lemma}
	\begin{proof}
		Let $\bfT = (T_{1} ,\dots ,T_{d})$ be a pure $1/k$-contraction. By Theorem \ref{V_T}, $\bfT$ is unitarily equivalent to the compression of $\bfM_{\bm{z}}^{(k)} \otimes I_{\overline{\text{Ran}} \Delta_{\bfT}}$ to a co-invariant subspace. Since $\bfM_{\bm{z}}^{(k)} \otimes I_{\overline{\text{Ran}} \Delta_{\bfT}}$ is a pure $1/s$-contraction (by Lemma \ref{spure}), its compression to a co-invariant subspace is also a pure $1/s$-contraction. Therefore, $\bfT$ is a pure $1/s$-contraction.
	\end{proof}
	
	For a pure $1/k-$contraction $\bfT = (T_{1}, \dots , T_{d}),$ we define 
	\begin{equation}\label{Gamma_T}
	\Gamma_{\bfT} = \left( I - \sum_{\alpha\in\mathbb{Z}^{d}_{+} \backslash \{0\}} b_{\alpha}^{(s)} \bfT^{\alpha} (\bfT^{\alpha})^{*}\right)^{1/2}.
	\end{equation}
	
	\subsection{The Taylor joint spectrum and functional calculus}
	The Taylor joint spectrum $\sigma(\bfT)$ for a tuple of commuting bounded operators $\bfT = (T_{1}, \dots, T_{d})$ is a compact subset of $\mathbb{C}^{d}$ which arises out of a certain Koszul complex and carries nice properties like an analytic functional calculus, for example. We shall not go into the definition here but refer the reader to Definition 14 in \cite{Mue}.
	\begin{lemma}
		Let $\bfT = (T_{1}, \dots, T_{d})$ be a pure $1/k$-contraction. Then $\sigma(\bfT) \subseteq \overline{  \mathbb{B}_{d}}.$
	\end{lemma}
	\begin{proof}
		By assumption, the power series $\sum_{n=0}^{\infty} a_{n}^{(s)} t^{n}$ has a radius of convergence of $1$. According to Lemma \ref{k implies s}, the $d$-tuple $\bfT$ is a $1/s$-contraction. The remainder of the proof follows from Lemma 5.3 in \cite{CH}.
	\end{proof}
	For a pure $1/k$-contraction $\bfT = (T_{1}, \dots, T_{d})$ acting on a Hilbert space $\cH$, we define the following notations for $\bm{w} \in \mathbb{B}_{d}$:
	$$k_{\bm{w}}(\bfT) = \sum\limits_{\alpha\in\mathbb{Z}^{d}_{+}} a_{\alpha}^{(k)} \overline{\bm{w}^{\alpha}} \bfT^{\alpha}, \hspace{3mm} s_{\bm{w}}(\bfT) = \sum\limits_{\alpha\in\mathbb{Z}^{d}_{+}} a_{\alpha}^{(s)} \overline{\bm{w}^{\alpha}} \bfT^{\alpha}, \hspace{3mm} g_{\bm{w}}(\bfT) = \sum\limits_{\alpha\in\mathbb{Z}^{d}_{+}} a_{\alpha}^{(g)} \overline{\bm{w}^{\alpha}} \bfT^{\alpha}.$$
	All three series $k_{\bm{w}}(\bfT)$, $s_{\bm{w}}(\bfT)$, and $g_{\bm{w}}(\bfT)$ converge in the norm operator topology. This can be established using the multi-variable functional calculus, as stated in Theorem III.9.9 in \cite{Vas}. By applying Theorem III.9.9 in \cite{Vas} once again, we obtain the equality
	\begin{equation*}
		k_{\bm{w}}(\bfT) = s_{\bm{w}}(\bfT) g_{\bm{w}}(\bfT)
	\end{equation*}
	and we also conclude that the series $\sum_{\alpha\in\mathbb{Z}^{d}_{+} \backslash \{0\}} b_{\alpha}^{(s)} \overline{\bm{w}^{\alpha}} \bfT^{\alpha}$ converges in the norm operator topology. Furthermore, we have the identity
	\begin{equation}\label{id3}
		\Big(  I - \sum\limits_{\alpha\in\mathbb{Z}^{d}_{+} \backslash \{0\}} b_{\alpha}^{(s)} \overline{\bm{w}^{\alpha}} \bfT^{\alpha} \Big)^{-1} = \sum\limits_{\alpha\in\mathbb{Z}^{d}_{+}} a_{\alpha}^{(s)} \overline{\bm{w}^{\alpha}} \bfT^{\alpha} = s_{\bm{w}}(\bfT).
	\end{equation}
	
	\subsection{Construction of the characteristic function}
	Let $\bfT = (T_{1} , \dots, T_{d})$ be a pure $1/k-$contraction acting on a Hilbert space $\cH.$ For each $\alpha\in\mathbb{Z}^{d}_{+},$ we define a subspace $\cE_{\alpha}$ of Hilbert space $\cH$ as follows:
	$$\cE_{\alpha} = \begin{cases} \overline{\text{Ran}} \Delta_{\bfT}, & a_{\alpha}^{(g)} \neq 0 \\ 0, & a_{\alpha}^{(g)} = 0\end{cases} .$$
	We define a new Hilbert space $\cE$ as the direct sum of the subspaces $\cE_{\alpha}$:
	$$\cE = \oplus_{\alpha\in\mathbb{Z}^{d}_{+}} \cE_{\alpha}.$$
	
	Since $a_{\alpha}^{(g)} \leq a_{\alpha}^{(k)}$ for all $\alpha \in \mathbb{Z}^{d}_{+},$ we have the following operator relation:
	\begin{equation}\label{id4}
		\sum\limits_{\alpha\in\mathbb{Z}^{d}_{+}} a_{\alpha}^{(g)} \bfT^{\alpha} \Delta_{\bfT}^{2} (\bfT^{\alpha})^{*} \leq \sum\limits_{\alpha\in\mathbb{Z}^{d}_{+}} a_{\alpha}^{(k)} \bfT^{\alpha} \Delta_{\bfT}^{2} (\bfT^{\alpha})^{*} = I_{\cH}
	\end{equation}
	where $I_{\cH}$ denotes the identity operator on the Hilbert space $\cH$. We define the linear map $\Pi_{\bfT} : \cH \to \cE  $  as follows:
	$$ \Pi_{\bfT} : h \mapsto  \begin{bmatrix}
		\vdots \\
		(a_{\alpha}^{(g)})^{1/2} \Delta_{\bfT} (\bfT^{\alpha})^{*} h\\
		\vdots
	\end{bmatrix}_{\alpha\in\mathbb{Z}^{d}_{+}}.$$
	By using \eqref{id4}, we can see that $\Pi_{\bfT}$ is a contraction.
	\begin{lemma} \label{Pi}
		The operator $\Pi_{\bfT} : \cH \to \cE$ defined above and the operator $\Gamma_{\bfT}: \cH \to \cH$ defined in \eqref{Gamma_T} satisfy the following:
		$$ \Pi_{\bfT}^{*} \Pi_{\bfT} = \Gamma_{\bfT}^{2}.$$
	\end{lemma}
	\begin{proof}
		For $\alpha\in\mathbb{Z}^{d}_{+} \backslash \{0\}$ we have the following relation: $$a_{\alpha}^{(g)} = a_{\alpha}^{(k)} - \sum\limits_{\gamma \in \mathbb{Z}^{d}_{+} \backslash \{0\}} b_{\gamma}^{(s)} a_{\alpha - \gamma}^{(k)}.$$
		We can compute $\Pi_{\bfT}^* \Pi_{\bfT}$ as follows:
		\begin{align*}
			\Pi_{\bfT}^{*} \Pi_{\bfT}
			& = \sum\limits_{\alpha \in\mathbb{Z}^{d}_{+}}  a_{\alpha}^{(g)} \bfT^{\alpha} \Delta_{\bfT}^{2} (\bfT^{\alpha})^{*} \\
			& = \Delta_{\bfT}^{2} +  \sum\limits_{\alpha \in\mathbb{Z}^{d}_{+} \backslash \{0\}} \Big(a_{\alpha}^{(k)} -  \sum\limits_{\gamma \in\mathbb{Z}^{d}_{+} \backslash \{0\}} b_{\gamma}^{(s)} a_{\alpha - \gamma}^{(k)} \Big) \bfT^{\alpha} \Delta_{\bfT}^{2} (\bfT^{\alpha})^{*} \\
			& = I - \sum\limits_{\gamma \in\mathbb{Z}^{d}_{+} \backslash \{0\}} b_{\gamma}^{(s)} \sum\limits_{\alpha \in\mathbb{Z}^{d}_{+} } a_{\alpha - \gamma }^{(k)} \bfT^{\alpha} \Delta_{\bfT}^{2} (\bfT^{\alpha})^{*} \\
			& = I - \sum\limits_{\gamma \in\mathbb{Z}^{d}_{+} \backslash \{0\}} b_{\gamma}^{(s)} \bfT^{\gamma} \Big( \sum\limits_{\alpha \in\mathbb{Z}^{d}_{+} } a_{\alpha }^{(k)} \bfT^{\alpha} \Delta_{\bfT}^{2} (\bfT^{\alpha})^{*} \Big)  (\bfT^{\gamma})^{*} \\
			& = I - \sum\limits_{\gamma \in\mathbb{Z}^{d}_{+} \backslash \{0\}} b_{\gamma}^{(s)} \bfT^{\gamma} (\bfT^{\gamma})^{*} .
		\end{align*}
		The last quantity is $\Gamma_{\bfT}^{2}$.
	\end{proof}

	By Lemma \ref{Pi}, there exists a unique unitary operator $u: \overline{\rm Ran} \Gamma_{\bfT} \to \overline{\rm Ran} \Pi_{\bfT}$ with $ u \Gamma_{\bfT} h = \Pi_{\bfT} h $ for $h\in\cH $. We define two Hilbert spaces: $\hat{\cH} = \cE \ominus \overline{\rm Ran} \Pi_{\bfT}$ and $\tilde{\cH} = \oplus_{\alpha\in\mathbb{Z}^{d}_{+} \backslash \{0\}} \cH.$ For each multi-index $\alpha\in\mathbb{Z}^{d}_{+} \backslash \{0\}$, let $\psi_{\alpha} : \mathbb{B}_{d} \to \mathbb{C}$ be the polynomial defined by $\psi_{\alpha}(\bm{z}) = (b_{\alpha}^{(s)})^{1/2} \bm{z}^{\alpha}.$ We define the infinite operator tuple $\bfZ : \tilde{\cH} \to \cH$ by
	$$\bfZ = (\psi_{\alpha}(\bm z) I_{\cH})_{\alpha\in\mathbb{Z}^{d}_{+} \backslash \{0\}}.$$
	The operator $\bfZ$ is a strict contraction because
	$$\|\bfZ\|^{2} = \sum\limits_{\alpha\in\mathbb{Z}^{d}_{+} \backslash \{0\}} b_{\alpha}^{(s)} | \bm{z}^{\alpha}|^{2} = 1 - \frac{1}{s(\bm{z}, \bm{z})} < 1.$$
	Let $\tilde{\bfT}: \tilde{\cH} \to \cH$ be the infinite operator tuple
	\begin{equation*}\label{def Ttilde}
		\tilde{\bfT}  =  (\psi_{\alpha}(\bfT))_{\alpha\in\mathbb{Z}^{d}_{+} \backslash \{0\}}.
	\end{equation*}
	It can be verified that $\tilde{\bfT}$ is a contraction if and only if $\bfT$ is a $1/s$-contraction.  Furthermore, we have
	$$\tilde{\bfT} \tilde{\bfT}^{*} = \sum\limits_{\alpha\in\mathbb{Z}^{d}_{+} \backslash \{0\}} b_{\alpha}^{(s)} \bfT^{\alpha} (\bfT^{\alpha})^{*} =  I_{\cH} - \Gamma_{\bfT}^{2},$$
	which implies $\Gamma_{\bfT}^{2} = I_{\cH} - \tilde{\bfT} \tilde{\bfT}^{*}.$ Let $D_{\tilde{\bfT}}$ be the unique positive square root of the positive operator $I_{\tilde{H}} - \tilde{\bfT}^{*} \tilde{\bfT},$ and let $\cD_{\tilde{\bfT}} = \overline{\rm Ran} D_{\tilde{\bfT}}.$ Using equation (I.3.4) from \cite{NF}, we have the identity
	\begin{equation}\label{defect}
		\tilde{\bfT}D_{\tilde{\bfT}}= \Gamma_{\bfT} \tilde{\bfT}.
	\end{equation}
	Since $\bfZ \tilde{\bfT}^{*}$ is a strict contraction, the operator $I_{\cH}-\bfZ \tilde{\bfT}^{*}$ is invertible. Moreover, using the definition of $\bfZ$ and $\tilde{\bfT}$, we have
	$$ I_{\cH} - \bfZ \tilde{\bfT}^{*} =  I_{\cH} - \sum\limits_{\alpha\in\mathbb{Z}^{d}_{+} \backslash \{0\}} b_{\alpha}^{(s)} \bm{z}^{\alpha}(\bfT^{\alpha})^{*}.$$
	From equation \eqref{id3}, we know that
	$$(I_{\cH} - \bfZ \tilde{\bfT}^{*})^{-1} = s_{\bm z}(\bfT)^{*}.$$
	This implies that for any $\bm z\in\mathbb{B}_{d},$ we have
	\begin{equation}  \label{I4}
		(g_{\bm{z}}(\bfT))^{*} = (k_{\bm{z}}(\bfT))^{*} (I_{\cH} - \bfZ \tilde{\bfT}^{*}).
	\end{equation}
	Let us define two linear maps:
	\begin{enumerate}
		\item $B : \cD_{\tilde{\bfT}} \oplus \hat{\cH} \to \tilde{\cH}$ given by the row matrix representation
		$ B \bydef \begin{bmatrix}
			D_{\tilde{\bfT}} & 0
		\end{bmatrix}.$
		\item $ D: \cD_{\tilde{\bfT}} \oplus \hat{\cH} \to \cE $ by the row matrix representation
		$ D \bydef \begin{bmatrix}
			-u \tilde{\bfT} & -i_{\hat{\cH}}
		\end{bmatrix}$
		where $i_{\hat{\cH}} : \hat{\cH} \to \cE $ is the inclusion mapping.
	\end{enumerate}
	We can define the map $D$ because of the identity in \eqref{defect}, which guarantees that $\tilde{\bfT}$ maps $\cD_{\tilde{\bfT}}$ into $\overline{\rm Ran} \Gamma_{\bfT}$, which is the domain of $u$. It may be convenient to represent $D$ as the column matrix
	$$D = \begin{bmatrix}
		\vdots \\
		D_{\alpha}\\
		\vdots
	\end{bmatrix}_{\alpha\in\mathbb{Z}^{d}_{+}}$$ with respect to the decomposition $\cE = \oplus_{\alpha \in\mathbb{Z}^{d}_{+}} \cE_{\alpha}.$

	\begin{definition}
		The characteristic function of a pure $1/k-$contraction $\bfT = (T_{1}, \dots, T_{d})$ is the analytic operator-valued function $\theta_{\bfT}: \mathbb{B}_{d} \to \cB(\cD_{\tilde{\bfT}} \oplus \hat{\cH}, \overline{\rm Ran} \Delta_{\bfT})$ defined by
		\begin{equation} \label{chfn2}
			\theta_{\bfT}(\bm{z}) = \Bigl( \sum\limits_{\alpha\in\mathbb{Z}^{d}_{+}} (a_{\alpha}^{(g)})^{1/2} D_{\alpha} \bm{z}^{\alpha} \Bigr) + \Delta_{\bfT} (k_{\bm{z}}(\bfT))^{*} \bfZ B .
		\end{equation}
	\end{definition}
	
	It is easy to see that the series $ \sum\limits_{\alpha\in\mathbb{Z}^{d}_{+}} (a_{\alpha}^{(g)})^{1/2} D_{\alpha} \bm{z}^{\alpha} $ converges strongly for all $\bm z\in\mathbb{B}_{d}.$ 

\begin{remark}
The choice of $\cE$ is not unique. The choice of $\cE_\alpha$ corresponding to those indices $\alpha$ in $\mathbb{Z}^{d}_{+}$ for which $a_{\alpha}^{(g)} = 0$ do not affect subsequent computations. 
Since we obviously want the characteristic function $\theta_{\bfT}$ to boil down to what it was in Definition 4.7 of \cite{BJ} in the special case $k=s$,  we need to ensure that $\hat{\cH}$ is the zero space in that case. This is ensured by the particular choice of $\cE$ which we made.
\end{remark}

We claim that $\theta_{\bfT}$ is a partially isometric multiplier from $H_{s}(\cD_{\tilde{\bfT}} \oplus \hat{\cH})$ to $H_{k}(\cE).$ To show that, we first proceed to compute $\theta_{\bfT}(\bm{z}) \theta_{\bfT}(\bm{w})^{*}$ for $\bm{z}, \bm{w} \in \mathbb{B}_{d}.$ The proof below highlights the separate roles of $s$ and $g$ following the main ideas from \cite{BDS}.
	
	\begin{lemma}\label{lemma1}
		For $\bm{z}, \bm{w} \in \mathbb{B}_{d},$ the identity
		\begin{equation*}
			s(\bm{z} , \bm{w}) \theta_{\bfT}(\bm{z}) \theta_{\bfT}(\bm{w})^{*} = k(\bm{z},\bm{w}) I_{\overline{\rm Ran} \Delta_{\bfT}}  - \Delta_{\bfT} k_{\bm{z}}(\bfT)^{*} k_{\bm{w}}(\bfT) \Delta_{\bfT}
		\end{equation*}
		holds.
	\end{lemma}
	\begin{proof}
		Fix $\bm z, \bm w \in \mathbb{B}_{d}.$ For simplicity, let us set the following notations:
		$$ x_{\alpha} = (a_{\alpha}^{(g)})^{1/2}, \quad \text{and}$$
		$$X(\bm{z}) = \sum\limits_{\alpha\in\mathbb{Z}^{d}_{+}} x_{\alpha} D_{\alpha} \bm{z}^{\alpha}, \quad  Y(\bm{z}) =  \Delta_{\bfT} (k_{\bm{z}}(\bfT))^{*} \bfZ B.$$
		With these notations we can write the following:
		$$ \theta_{\bfT}(\bm{z}) \theta_{\bfT}(\bm{w})^{*} = X(\bm{z}) X(\bm{w})^{*} + X(\bm{z}) Y(\bm{w})^{*} + Y(\bm{z}) X(\bm{w})^{*} + Y(\bm{z}) Y(\bm{w})^{*}.$$
		A direct computation shows that the operator
		$$ U \bydef \begin{bmatrix}
			\tilde{\bfT}^{*} & B \\
			\Pi_{\bfT} & D
		\end{bmatrix} : \cH \oplus (\cD_{\tilde{\bfT}} \oplus \hat{\cH}) \to \tilde{\cH} \oplus \cE $$
		defines a block unitary matrix. Using the fact that $\Pi_{\bfT} \Pi_{\bfT}^{*} + D D^{*} = I_{\cE }$ and identity \eqref{I4}, we have
		$$	X(\bm{z})	X(\bm{w})^{*} = g(\bm{z}, \bm{w})I_{\overline{\text{Ran}} \Delta_{\bfT}}- \Delta_{\bfT} (k_{\bm{z}}(\bfT))^{*} (I - \bfZ \tilde{\bfT}^{*}) (I - \tilde{\bfT} \bfW^{*}) k_{\bm{w}}(\bfT) \Delta_{\bfT} .$$
		We use the equation $\Pi_{\bfT} \tilde{\bfT} + D B^{*} = 0 $ and identity \eqref{I4} to obtain
		$$X(\bm{z}) Y(\bm{w})^{*} = - \Delta_{\bfT} (k_{\bm{z}}(\bfT))^{*} (I - \bfZ \tilde{\bfT}^{*}) \tilde{\bfT} \bfW^{*} k_{\bm{w}}(\bfT) \Delta_{\bfT} .$$
		Similarly, we get that
		$$ Y(\bm{z}) X(\bm{w})^{*} = - \Delta_{\bfT} (k_{\bm{z}}(\bfT))^{*} \bfZ \tilde{\bfT}^{*} (I - \tilde{\bfT} \bfW^{*}) k_{\bm{w}}(\bfT) \Delta_{\bfT} .$$
		Finally, we use $\tilde{\bfT}^{*} \tilde{\bfT} + B B^{*} = I_{\tilde{\cH}}$ to obtain
		$$ Y(\bm{z}) Y(\bm{w})^{*}  =  \Delta_{\bfT} (k_{\bm{z}}(\bfT))^{*} \bfZ (I - \tilde{\bfT}^{*} \tilde{\bfT}) \bfW^{*} k_{\bm{w}}(\bfT) \Delta_{\bfT}.$$
		Combining these expressions, we obtain the desired identity.
	\end{proof}

	The proof of the following lemma follows the same approach as the proof of Lemma 4.9 in \cite{BJ}.
	
	\begin{lemma}\label{V_T*2}
		For any $\bm{w}\in\mathbb{B}_{d}$ and $\xi \in \overline{{\rm Ran}} \Delta_{\bfT}$ we have the identity
		$$V_{\bfT}^{*}(k_{\bm{w}}\otimes\xi)=k_{\bm{w}}(\bfT) \Delta_{\bfT}\xi.$$
	\end{lemma}
	
	\begin{thm}\label{main}
		Let $\bfT = (T_{1},\dots,T_{d})$ be a pure $1/k$-contraction. Then its characteristic function $\theta_{\bfT}$ defined by \eqref{chfn2} induces a partially isometric multiplier
		$$M_{\theta_{\bfT}} : H_{s} \otimes (\cD_{\tilde{\bfT}} \oplus \hat{\cH}) \to H_{k} \otimes \overline{\rm Ran} \Delta_{\bfT}$$
		such that
		\begin{equation}  \label{NewIdentity}
			V_{\bfT} V_{\bfT}^{*}+M_{\theta_{\bfT}}M_{\theta_{\bfT}}^{*}=I.
		\end{equation}
	\end{thm}
	\begin{proof}
		Define a linear map
		$$A : \text{span}\{k_{\bm{w}} \otimes \xi : \bm{w}\in\mathbb{B}_{d}, \xi\in \overline{\text{Ran}} \Delta_{\bfT} \} \to H_{s} \otimes (\cD_{\tilde{\bfT}} \oplus \hat{\cH})$$
		given by
		$$A (k_{\bm{w}}\otimes\xi)=s_{\bm{w}}\otimes\theta_{\bfT}(\bm{w})^{*}\xi. $$
		Using Lemma \ref{lemma1} we get
		$$\la A (k_{\bm{w}}\otimes\xi) , A (k_{\bm{z}}\otimes\eta) \ra
		=  \la k_{\bm{w}} \otimes \xi , k_{\bm{z}} \otimes \eta \ra - \la V_{\bfT}^{*} (k_{\bm{w}} \otimes \xi) , V_{\bfT}^{*} (k_{\bm{z}} \otimes \eta) \ra. $$
		Now we use Lemma \ref{V_T*2} to obtain
		\begin{equation}
			\la A (k_{\bm{w}}\otimes\xi) , A (k_{\bm{z}}\otimes\eta) \ra = \big\la \big( I - V_{\bfT} V_{\bfT}^{*} \big) (k_{\bm{w}} \otimes \xi) , k_{\bm{z}} \otimes \eta \big\ra \label{id}.
		\end{equation}
		This implies that $$ \| A x\| \leq \|x\|$$ for all $x \in {\rm span}\{k_{\bm{w}} \otimes \xi : \bm{w}\in\mathbb{B}_{d}, \xi\in \overline{\rm Ran} \Delta_{\bfT} \}.$ Thus, $A$ extends as a bounded operator from $H_{k} \otimes \overline{\rm Ran} \Delta_{\bfT}$ to  $H_{s} \otimes (\cD_{\tilde{\bfT}} \oplus \hat{\cH})$ and in fact becomes $M_{\theta_{\bfT}}^{*}$. The proof of \eqref{NewIdentity} follows from \eqref{id}.
	\end{proof}
	
	\begin{remark}\label{K-inner}
	Although $M_{\theta_{\bfT}}$ is a partial isometry, there is a maximal closed subspace of the space of ``constant functions'' on which it acts isometrically. More precisely, let 
	$$\cM = \{ x \in \cD_{\tilde{\bfT}} \oplus \hat{\cH} : \| M_{\theta_{\bfT}} x \| = \|x\| \}.$$ 
	This space is always non-trivial.  Moreover, the function $\varphi_{\bfT}: \mathbb{B}_{d} \to \cB(\cM , \overline{\rm Ran}\Delta_{\bfT})$ defined by $$\varphi_{\bfT}(\bm z) = \theta_{\bfT}(\bm z)|_{\cM} $$ 
	is a $k-$inner function as defined in the introduction.
	\end{remark}
	
	We end by noting that the usual notion of coincidence of characteristic functions can be formulated in this context as well. The characteristic functions of two unitarily equivalent $1/k$-contractions clearly coincide. In the case of pure $1/k$-contractions, the converse can be established by constructing a functional model. Using the isometry $ V_{\bfT}$ as well as the identity \eqref{NewIdentity}, we get that every pure $1/k$-contraction $\bfT = (T_{1},\hdots,T_{d})$ acting on a Hilbert space $\cH$ is unitarily equivalent to the commuting tuple $\mathbb{T} = ( \mathbb{T}_{1}, \hdots,\mathbb{T}_{d})$ on the functional space
	$$\mathbb{H}_{\bfT}=(H_{k} \otimes \overline{\rm Ran}\Delta_{\bfT}) \ominus {\rm Ran} M_{\theta_{\bfT}}$$
	defined by $\mathbb{T}_{i}=P_{\mathbb{H}_{\bfT}}(M_{z_{i}}^{(k)} \otimes I_{\overline{\rm Ran} \Delta_{\bfT}})|_{\mathbb{H}_{\bfT}}$ for $1 \leq i \leq d.$  The functional model yields the following theorem.
	
	\begin{thm}
		Two pure $1/k$-contractions are unitarily equivalent if and only if their characteristic functions coincide.
	\end{thm}
	
	\begin{remark}
	Section 4 in Eschmeier' work \cite{Esch} has been an overarching influence while tring to conceive the present work. The generalized Bergman kernels considered by Eschmeier are examples of the kernel $k$. For this class of examples, results analogous to our Lemma \ref{V_T*2} and Theorem \ref{main} were observed in the first display equation of page 97 and in Theorem 14 respectively. An identity analogous to Lemma \ref{lemma1} was proved in the proof of Theorem 14. In a remark preceding his Theorem 14, Eschmeier explained the connection between characteristic function and $K_{m}-$inner functions.  This motivated our Remark \ref{K-inner}.
	\end{remark}
	Acknowledgments.
	The first named author thanks Jaydeb Sarkar for a brief but insightful conversation. His work is supported by a J C Bose Fellowship JCB/2021/000041 of SERB. The second named author thanks Michael Hartz for his comments on Theorem \ref{s_1 = s_2}. He also thanks Sebastian Toth for a discussion on $k-$inner functions. His work is supported by the Prime Minister’s Research Fellowship PM/MHRD-20-15227.03. This work is also supported by the DST FIST program - 2021 [TPN - 700661]. We thank the anonymous referee for useful suggestions.

\end{document}